\documentclass[14pt]{article}

\usepackage[utf8]{inputenc}
\usepackage[T2A]{fontenc}
\usepackage{amsmath}
\usepackage{amssymb}
\usepackage{amsthm}
\usepackage{cmap}
\usepackage{mathrsfs}
\usepackage{tikz-cd}
\usepackage{comment}
\usepackage[symbol]{footmisc}
\usepackage{hyperref}
\hypersetup{
    colorlinks=true,
    linkcolor=blue,
    filecolor=magenta,      
    urlcolor=cyan,
}
\def\R{{\mathbb R}}
\def\C{{\mathbb C}}
\def\Q{{\mathbb Q}}
\def\Z{{\mathbb Z}}
\def\F{{\mathbb F}}

\def\Der{{\hbox{\bf Der}}}
\def\Inn{{\hbox{\bf Inn}}}
\def\Out{{\hbox{\bf Out}}}

\def\Hom{{\hbox{\bf Hom}}}
\def\Ob{{\hbox{\bf Ob}}}

\def\Out{{\hbox{\bf Out}}}

\def\cH{{\cal H}}

\newtheorem{theorem}{Theorem}[section]
\newtheorem{lemma}{Lemma}[section]

\newtheorem{proposition}{Proposition}[section]
\newtheorem{definition}{Definition}[section]
\newtheorem{corollary}{Corollary}[theorem]

\newtheorem*{fact}{Statement}
\DeclareMathOperator{\im}{im}
\makeatletter
\newcommand{\bigplus}{%
  \DOTSB\mathop{\mathpalette\mattos@bigplus\relax}\slimits@
}
\newcommand\mattos@bigplus[2]{%
  \vcenter{\hbox{%
    \sbox\z@{$#1\sum$}%
    \resizebox{!}{0.9\dimexpr\ht\z@+\dp\z@}{\raisebox{\depth}{$\m@th#1+$}}%
  }}%
  \vphantom{\sum}%
}
\makeatother

\title{Derivations of group rings for finite and FC groups}
\author{A.A.Arutyunov\footnote{Department of Higher Mathematics, Moscow Institute of Physics and Technology, andronick.arutyunov@gmail.com}\footnote{V.A. Trapeznikov Institute of Control Sciences of RAS,
65 Profsoyuznaya str., Moscow 117997, Russia};
L.M.Kosolapov\footnote{Higher School of Economics ,  Moscow, 101000, Russia, 
leokosol@gmail.com}
\footnote{Skolkovo Institute of Science and Technology, Moscow, 121205, Russia
}}

\begin{document}

\maketitle
\begin{abstract}
In this paper we establish decomposition theorems for derivations of group rings. We provide a topological technique for studying derivations of a group ring $A[G]$ in case $G$ has finite conjugacy classes.  As a result, we describe all derivations of algebra $A[G]$ for the case when $G$ is a finite group, or $G$ is an FC-group.   In addition, we describe an algorithm to explicitly calculate all derivations of a group ring $A[G]$ in case $G$ is finite. As examples, derivations of $\Z_4[S_3]$ and $\F_{2^m}D_{2n}$ are considered. 
\end{abstract}
% We also provide a structure theorem for inner derivations in case $A$ is a PID (principal ideal domain) and $G$ is finite.

\section{Introduction}

% Решаем деривационную проблем для конечных групповых алгебр. Оказывается, что есть случаи когда внешние деривации нетривиальны, есть случаи когда они тривиальны. Данный вопрос решается.

% Изучаем конечные группы и их естественное обобщение FC-группы. И строим внешнюю алгебру дериваций, в частности даем её конструктивное описание. 

In this paper we solve the derivation problem for group algebras of finite groups. Our results provide constructive description of derivation algebra $\Der(\F[G])$ for the case when $\F$ is a finite field and $G$ is a finite group. We discuss examples when outer derivations algebra is nontrivial, and provide some criterions for its triviality.

Our theorems admit a natural generalization to the case when $\F$ is an arbitrary commutative unital ring, and $G$ is an FC-group. In these cases our technique provides decomposition theorems for derivation algebra.

Our method opens the way for practical applications of derivation algebras. As an example, the explicit description of derivation algebras is useful for construction of binary codes (see \cite{Creedon-Hughes}).
\subsection{History of the topic and motivation for this research}

In this paper we apply the the character technique approach, which is proposed in  \cite{Arutyunov_1}, to the case when group has finite conjugacy classes. As a consequence, the character technique gives tools to prove a decomposition theorem for $\Der(A[G])$ in case $G$ is an FC-group and $A$ is a commutative unital ring. As a corollary, we prove a theorem describing all inner and outer derivations of a group ring of a finite group. In addition, theorems about character complexes for finite conjugacy classes may be useful for studying Derivation problem for groups which have nontrivial finite conjugacy classes.

Derivations of group rings have been a topic for studies since late 1970s.
Smith (see \cite{Smith}) was one of the first to study the derivations in group rings. In certain papers (see \cite{Burkov, Ferrero_Giambruno_Milies,  Spiegel}) were studied the properties of group rings which have no outer derivations. 

There is a lot of papers concerning derivations of group rings: derivations of group rings and polynomial rings were studied at \cite{Bavula_2, Burkov, Burkov_2,Kharchenko_1, Kiss_Laczkovich}, derivations of particular group algebras (incindence, Grassman, Novikov, nilpotent algebras, etc.) were studied in \cite{ Bai_Meng_He, Bavula_1, Kaygorodov_Popov_1}, and there is a lot of topics connected with derivations of group rings (see \cite{ Lanski_Montgomery, Spiegel}). In recent papers \cite{ Artemovych_2,Artemovych_5, Artemovych_4}, properties of generalized inner derivations, central derivations and derivations of group rings over finite rings were studied.

% To understand the structure of derivation algebra the derivation problem plays an important role. The derivation problem is whether all derivation in a group ring are inner. This question for Banach algebra $L_1(G)$ was formulated in \cite{Dales} (Q5.6.B, p.746). Results for some special cases can be found in \cite{Johnson-2001,Losert-2008}.

The derivation problem asks, in essense, whether all derivations in a group ring are inner. This question for Banach algebra $L_1(G)$ was formulated in \cite{Dales} (Q5.6.B, p.746). Results for some special cases can be found in \cite{Johnson-2001,Losert-2008}.

In recent papers \cite{Arutyunov_1,Arutyunov_2,Arutyunov_3,Arutyunov_4, Arutyunov_5,Mis} a categorical method of studying derivations was introduced. The method is based upon a correspondence between derivations and characters (additive functions on groupoid of adjoint action of a group). This method proved to be useful in studying derivation problem (i.e. finding whether outer derivation space is trivial). 

  Application of characters technique for these problems is described in \cite{Arutyunov_1, Arutyunov_2, Arutyunov_3}. Moreover, in recent paper \cite{Arutyunov_5}  character technique proved to be useful for studying $(\sigma,\tau)$-derivations. Derivations, and especially $(\sigma,\tau)$-derivations of group rings have various applications in coding theory \cite{ Boucher-Ulmer, Creedon-Hughes}.

In general, information on group rings, derivations, endomorphisms of group rings, Lie rings (Lie algebras over rings) and Lie algebras can be found in \cite{Bovdi_book,Fuchs_book,Herstein_book,Jacobson, Robinson_book}.

In our paper we prove the decomposition theorems which provide the way to explicitly describe all derivations of a finite group ring. This can be useful for coding theory applications (see \cite{ Boucher-Ulmer, Creedon-Hughes}). The decomposition theorems we prove also give rise to the criterion for solving the problem of outer derivations' existence for finite group rings. In addition, our technique is applicable to FC-group rings, which can be useful for studying their derivations.

\subsection{Main results}

Let $G$ be a noncommutative group, $A$ - commutative unital ring. We study the $A$-module of derivations over a group ring $A[G]$. Note that group ring $A[G]$ comprises of finite sum of elements from $G$ with coefficients from $A$. Thus, endomorphisms of $A[G]$ are linear operators, which map finite sums to finite sums.
\begin{definition}
Operator $d:A[G]\to A[G]$ is called an endomophism of group ring $A[G]$, if:
\begin{itemize}
\item $  d(x+y)=d(x)+d(y)\text{ for } \forall x,y\in A[G],$
\item $ r\in A, d(rx)=rd(x)\text{ for }\forall x\in A[G].$
\end{itemize}
\end{definition}

Of course, by the definition of $A[G]$, any such operator $d$ maps finite sums (of group's elements) to finite sums. 

\begin{definition}
Derivations of $A[G]$ are defined as endomorphisms of $A[G]$ that are subject to the Leibnitz rule. I.e. $d\in \Der(A[G])$, if $d$ is an endomorphism of $A[G]$ and
$$d(xy)=d(x)y+xd(y)\text{ for }\forall x,y\in A[G]$$
\end{definition}

\begin{definition}
Let $a\in A[G]$. Adjoint derivation $\textbf{ad}_a$ acts on any $x\in A[G]$ in the following way:
$$ \textbf{ad}_a(x)= [a,x]\equiv ax-xa $$Inner derivations are defined as usual:
	$$\Inn= \{d =\sum_{a \in G}C_a \textbf{ad}_a \text{ - finite sum }\}$$
    \end{definition}

It is a well-known fact (see \cite{Jacobson}) that $(\Der, [.,.])$ form a Lie algebra with respect to commutator, and $\Inn$ is an ideal in $\Der$. This makes quotient Lie algebra correctly defined:
    $$\Out:=\Der/\Inn.$$

In the article we consider the case when $G$ is a group with finite conjugacy classes -- so called FC-group (see \cite{Robinson_book}). In particular $G$ is a finite group. Also we will consider that $A$ is an unital commutative ring. 

{\bf Theorem \ref{th-FC-decomp}.}
{\it
Let $G$ be a finitely generated FC-group, $A$ be a unital commutative ring. Then 
\begin{equation*}
\Der(A[G])\cong \Inn(A[G])\oplus \bigoplus_{[u]\in G^G} \Hom_{Ab}(Z(u),A).
\end{equation*}}

Here (see definition \ref{def-Hom_ab}) $\Hom_{Ab}$ is the set of additive homomorphisms from the centralizer $Z(u)$ of fixed element $u\in G$ to the ring $A$.

This result can be specified for description of outer derivation whether $G$ is a finite group.

{\bf Theorem \ref{finite_result}.}{\it
There is a way to describe all derivations of $A[G]$ in case $G$ is finite:
\begin{enumerate}
\item $\Der(A[G])\cong\Inn(A[G]) \oplus \big(\bigoplus_{[u]\in G^G}   \Hom_{Ab}(Z(u), A)\big),$
\item $\Out(A[G])\cong \bigoplus_{[u]\in G^G}   \Hom_{Ab}(Z(u), A).$
\end{enumerate}}

Generally speaking algebra of outer derivations isn't trivial. But for special cases spaces $\Hom_{Ab}(Z(u),A)$ can be degenerated. 

First case is easily deduced from well-known results.

{\bf Corollary \ref{corol-torsion-free}.} {\it
If $A$ is a torsion-free ring, $G$ is a finite group, then all derivations are inner $\Out(A[G])=0$. As a consequence, if $A=\Z,\Q,\R,\C$, then $\Out(A[G])=0$.}

But the second case apparently was not previously studied and important for applications.

If the ring $A$ is finite then as an abelian group the following decomposition holds
$$A\cong\Z_{p_1^{i_1}}\oplus ... \oplus\Z_{p_n^{i_n}}.$$
 Also for finite group $G$ we get  
$$G/[G,G]\cong\Z_{q_1^{j_1}}\oplus ... \oplus\Z_{q_m^{j_m}}$$be the primary decomposition of $G/[G,G]$.

{\bf Theorem \ref{th-deriv-problem}.}
{\it For finite group $G$ and finite ring $A$ all derivations $\Der(A[G])$ are inner ($\Der(A[G])=\Inn(A[G])$ )
if and only if  $\{p_1,...,p_n\}\cap\{q_1,...,q_m\}=\emptyset$.}

The results obtained can be used in applications, some of which are described in the section \ref{ex-app}.
In particular we will consider the case of an algebra $\Z_4[S_3]$ (see results of section \ref{ex-z4-s3}) and $\F_{2^m}D_{2n}$ (see section \ref{ex-F2m-D2n}). The last example is important for applications especially for cryptography. So we will give a new combinatorial way to get such results from \cite{Creedon-Hughes,Conrad3,Conrad4}.

\section{Derivations of group rings}

\subsection{Groupoid of adjoint action and characters}

In this section we definte a category of adjoint action similarly to papers \cite{Arutyunov_1, Arutyunov_3}. A groupoid $\Gamma$ based on a noncommutative group G is defined in the following way:
\begin {itemize}
\item $\Ob(\Gamma) = G$
\item $\Hom (a, b) = \{(u, v)\in G \times G | v^{-1}u = a, uv^{-1} = b\}$ for every $a$, $b \in Obj(\Gamma)$
\item Composition of maps $\varphi = (u_1, v_1) \in$ \textbf{Hom} $(a, b)$, $\psi = (u_2, v_2) \in \Hom(b, c)$ is a map $\varphi \circ \psi \in \Hom(a, c)$ such that:
\begin{center}
	$\varphi \circ \psi = (u_2v_1, v_2v_1)$
\end{center}

\end {itemize}

Hereinafter for $x\in G$, the corresponding conjugacy class is $[x] = \{g x g^{-1}|g\in G\}$.

\begin{definition}
Define $\Gamma_{[u]}$ as a subgroupoid of $\Gamma$, such that:
\begin{itemize}
\item $\Ob(\Gamma_{[u]})=[u]$
\item $\Hom(\Gamma_{[u]})=\{\varphi\in \Hom(\Gamma)|t(\varphi)\in [u]\}$.
\end{itemize}
In other words, $\Gamma_{[u]}$ is a full subcategory in $\Gamma$.
\end{definition}

Denote $G^G=\{[u_1],\ldots, [u_n]\}$ as a set of conjugacy classes in $G$.  

\begin{fact} It is easy to check by direct calculation that:
\begin{itemize}
\item Two objects are connected in $\Gamma$ iff they are conjugated elements of $G$;
\item For any morphism $\varphi\in \Hom(\Gamma)$, we have $t(\varphi)\in [u]\iff s(\varphi)\in [u]$;
\item  For any conjugacy class $[u]$ in $G^G$, the corresponding subgroupoid $\Gamma_{[u]}$ is a connected component in $\Gamma$;
\item Hence $\Gamma$ can be represented as a disjoint union
$$\Gamma = \bigsqcup_{[u] \in G^G} \Gamma_{[u]}.$$
\end{itemize}
\end{fact}
Proof of these facts can be found in previous papers \cite{Arutyunov_3, Arutyunov_5}. 

The next definition is important (see \cite{Arutyunov_1}, section 3).
\begin{definition}
A function $\chi:\Hom(\Gamma) \to A$, such that $$\chi(\varphi\circ\psi)=\chi(\varphi)+\chi(\psi),$$is called a character on $\Gamma$. We denote the space (actually, an $A$-module) of such characters on $\Gamma$ as $X_1(\Gamma)$.
\end{definition}
Generally, this is not the same as the space of locally-finite characters.
\subsection{Character modules and derivations}
Let $d$ be a derivation of $A[G]$. 

For any basis element $g$ in the group algebra $A[G]$, we have

\begin{equation}
d(g) = \sum\limits_{h\in G}d_g^h h
\end{equation}
where $d^g_h\in A$ -- coefficients that depend only on the derivation $d$.
An arbitrary element $u\in A[G]$ can be represented as a finite sum $u = \sum\limits_{g\in G}\lambda^g g$, with coefficients $\lambda^g\in A$. Then the element $d(u)$ can be represented as 
\begin{equation}
\label{eqDER}
  d(u) =\sum\limits_{g\in G}\lambda^g d(g)= \sum\limits_{g\in G}\sum\limits_{h\in G} \lambda^g d^h_g h.
\end{equation}

Following \cite{Arutyunov_1,Arutyunov_3} the correspondence between derivations and characters is given by the usual definition:
\begin{equation} 
\label{d^g_h}
  \chi_d ((u,v)) = d^u_v.
\end{equation}
If $d\in \Der$, then $\forall g, d(g)=\sum d_g^h h =\sum \chi((h,g)) h$ is a finite sum. So we can say, that a character $\chi$ corresponds to a derivation, if and only if  for each $g$ exists just a finite number of elements $h$, such that $\chi((h,g))\neq 0$. In other words, $\forall g,\chi((h,g))= 0 $ for cofinitely many $h$'s. Such a character is called a locally finite character.

\begin{definition}
Following papers \cite{Arutyunov_1, Arutyunov_3}, a 1-character $\chi$ such that for every $v\in G$
  $$
    \chi(x,v) = 0
  $$
for almost all $x\in G$ is called locally finite.
\end{definition}
In other words, a character $\chi$ is locally finite, if for any $g\in G$, there is only a finite number of elements $x\in G$, such that $\chi(x,v)\neq 0$.
\begin{definition}
Define $X_1^{fin}$ as a set of locally finite 1-characters on the $\Gamma$.

\end{definition}

As an $A$-module, $X_1^{fin}$ is obviously isomorphic to $\Der$. However, we need to establish a Lie algebra isomorphism between $X_1^{fin}$ and $\Der$. 

\begin{theorem} 
If $d_1,d_2\in \Der$, then their commutator $[d_1, d_2]$  corresponds to 1-character $\chi_{[d_1, d_2]} $, which can be calculated by this formula
\begin{equation}
	\chi_{[d_1, d_2]}(a, g)  = \sum\limits_{h\in G} \chi_{d_1}(a, h)\chi_{d_2}(h, g) - \chi_{d_2}(a, h)\chi_{d_1}(h, g).
\end{equation}
\end{theorem}
Denote $$\{\chi_{d_1}, \chi_{d_2}\}:=	\chi_{[d_1, d_2]}.$$
\begin{proof}
	Let $g \in G$, then we have
    	$$d_1(g) = \sum\limits_{h\in G} \chi_{d_1}(h, g)h,\quad
        d_2(g) = \sum\limits_{h\in G} \chi_{d_2}(h, g)h.$$
        We are going to calculate the commutator $[d_1, d_2] = d_1d_2 - d_2d_1$. It's easy to see that
    
        $$ d_1d_2(g) = \sum\limits_{h\in G} \chi_{d_2}(h, g)(\sum\limits_{a\in G} \chi_{d_1}(a, h)a),$$
        $$ d_2d_1(g) = \sum\limits_{h\in G} \chi_{d_1}(h, g)(\sum\limits_{a\in G} \chi_{d_2}(a, h)a).$$
    
    Changing the sum order in these expressions we get
    
    \begin{center}
        $[d_1, d_2](h) = \sum\limits_{a\in G} (\sum\limits_{h\in G} \chi_{d_2}(h, g)\chi_{d_1}(a, h) - \chi_{d_1}(h, g)\chi_{d_2}(a, h))a.$
    \end{center}
    
    The formula for $\{\chi_{d_1}, \chi_{d_2}\}(a, h)$ is a coefficient of $a$, i.e.
    
    \begin{center}
    	$\{\chi_{d_1}, \chi_{d_2}\}(a, g) = \sum\limits_{h\in G} \chi_{d_1}(a, h)\chi_{d_2}(h, g) - \chi_{d_2}(a, h)\chi_{d_1}(h, g)$.
    \end{center}
    
\end{proof}
This theorem is an analogue of Proposal 4.3 from \cite{Arutyunov_4}.
\begin{corollary}
  The formula \eqref{d^g_h} defines the canonical isomorphism
$$(\Der, [\cdot,\cdot])\cong(X_1^{fin},\{\cdot,\cdot \}).$$ 
\end{corollary}

\subsection{Inner derivations and their characters}
Recall that 	$$\Inn= \{d =\sum_{a \in G}C_a \textbf{ad}_a \text{ - finite sum }\}.$$

The next statement follows from \cite{Arutyunov_1}, Theorem 2.  
\begin{fact}
An adjoint derivation $\textbf{ad}_a: x\mapsto [a,x]$ corresponds to the character

% \begin{equation}
% \chi_{\textbf{ad}_a}(\varphi)=\begin{cases}
% 0, &  s(\varphi)=t(\varphi)=a\\
% 1, & t(\varphi) = a,s(\varphi) \neq a\\
% -1, &  s(\varphi)=a, t(\varphi)\neq a\\
% 0, & \text{else}
% \end{cases}
% \end{equation}
\begin{equation}
\chi_{\textbf{ad}_a}((h,g))=\begin{cases}
0, &  g^{-1}h=hg^{-1}=a\\
1, & hg^{-1} = a, g^{-1}h \neq a\\
-1, &   g^{-1}h=a,hg^{-1}\neq a\\
0, & \text{else}
\end{cases}
\end{equation}

\end{fact}

Now we can write down definitions of inner derivation submodule in terms of characters. 
\begin{definition}
	$$\Inn= \{d\in \Der| \chi_d=\sum_a C_a \chi_{\textbf{ad}_a} \text{ - finite sum}\}$$

\end{definition}
We can clearly see that character of any $d\in \Inn$ is trivial on loops. %(Nevertheless, the converse is false.) This has many interesting implications:
\begin{fact}\label{character_properties}
Here are some basic facts about how trivial-on-loops characters on $\Gamma$ behave:
\begin{itemize}
\item If a character on a groupoid is trivial on loops, then it is constant on any $\Hom(a,b)$;
\item If a character $\chi$ on a groupoid is trivial on loops, then it is a potential character, i.e. $\exists p: \Ob(\Gamma)\to A=\C: \chi= \partial p$. The latest term means that $\forall \varphi\in \Hom(\Gamma), \chi(\varphi)=p(t(\varphi))-p(s(\varphi))$;
\item For any identity morphism $(a,e)$ its character $\chi(a,e)=0$;
\item If $\alpha\in \Hom(a,a), \beta\in \Hom(b,b), \exists \varphi\in \Hom(a,b): \alpha=\varphi^{-1}\circ\beta\circ\varphi$, then $\forall \chi\in X_1$, $\chi(\alpha)=\chi(\beta)$.
\end{itemize}
These facts are all quite obvious and their proofs are neglected for the sake of conciseness.
\end{fact}

We have already said above that a trivial-on-loops character may not have its corresponding derivation be inner. As an example, in some cases (infinite $G$, $A=\C$) there are locally finite trivial-on-loops characters, which do not correspond to an inner derivation. Counterexamples can be found in \cite{Arutyunov_2} (result for nilpotent groups in section 3.2 and more specific result for Heisenberg group in section 3.3).

\begin{definition}
Let us define a submodule of derivations, such that their characters are trivial on loops:
$$\Inn^*= \{d\in \Der| \chi_d\text{ is trivial on loops} \}$$
We call such derivations \textbf{weakly inner}. \end{definition}
This definition was proposed in \cite{Arutyunov_5} (see section 2.5) where this type of derivations was called quasi-inner.

It's easy to see that all inner derivation are weakly inner
 $$\Inn\subseteq \Inn^*.$$
 Converse is not true (see \cite{Arutyunov_2}, results of section 3.3).
\begin{definition}

We introduce definitions of character modules consisting only of trivial-on-loops characters:

$$X_0=\{\chi\in X_1|\chi \text{ is trivial on loops}\};$$
$$X^{fin}_0=\{\chi\in X_1|\chi \text{ is locally finite and trivial on loops}\}.$$
    
\end{definition}
Immediately from definition we get that $(X_0^{fin}, \{\cdot,\cdot\})$ is a Lie algebra. And moreover $(X_0^{fin}, \{\cdot,\cdot\})\cong (\Inn^*,[\cdot,\cdot])$. 

The following theorem is an analogue of the Theorem 2.2. for semigroup algebras from \cite{Arutyunov_5}.

\begin{lemma}\label{weakly_inner}
These conditions on character are equivalent:
\begin{enumerate}
\item $\chi $ is trivial on loops;
\item $\chi $ is potential;
\item $\chi$ can be expressed as a (possibly infinite) sum of adjoint derivations' characters: $ \chi=\sum_{g \in G}C_g \chi_{\textbf{ad}_g}, C_g\in A$.
\end{enumerate}

\end{lemma}

\begin{proof}
\textit{(1$\implies$2)}

If $\chi $ is trivial on loops, then for any $[u]\in G^G$, fix $u_0\in [u]$, and choose any value of $p(u_0)$, for example, $p(u_0)=0$. Then we can define potential $p$ on $\Gamma_{[u]}$ through the following procedure:
if $\varphi: u_0\to u, u\in [u_0]$, then $p(u)=p(u_0)+\chi(\varphi)$. On one hand, if such $\varphi$ exists, then $p(u)$ does not depend on the choice of $\varphi$ due to properties of trivial on loops characters. On the other hand, $\exists g: u=gu_0g^{-1} $. Hence $\exists  (gu_0,g): u_0\to u$, which is an example of such $\varphi$ above.

As a result, we obtain $p$ defined on the whole $\Gamma$, and $\chi=\partial (p)$.

\textit{(2$\implies$3)}

We can clearly see now that $\forall \varphi \in \Hom(\Gamma)$
$$\chi_d(\varphi)=p(t(\varphi))\chi_{\textbf{ad}_{t(\varphi)}}(\varphi)+\chi_{\textbf{ad}_{s(\varphi)}}p(s(\varphi))=p(t(\varphi))-p(s(\varphi)).$$
This gives us the correspondence between $C_g$ and the values of $p$.

\textit{(3$\implies$1)}

$\forall g, \chi_{\textbf{ad}_{g}}$ is trivial on loops. Any linear combination of trivial-on-loops characters is trivial on loops.
\end{proof}

\begin{lemma}
$\Inn^*$ is an ideal in $\Der$.
\end{lemma}
\begin{proof}
At first, let $d\in \Der$, $d_0=\textbf{ad}_x$

It is a well-known fact that $[d, \textbf{ad}_g](x)=\textbf{ad}_{d(g)}(x)$
\begin{multline*}
[d, \textbf{ad}_g](x)=d( \textbf{ad}_g(x)) - \textbf{ad}_g( d(x))=d(gx-xg)- [g, d(x)]=\\
=d(g)x+gd(x)-d(x)g-xd(g)- [g, d(x)]=[d(g),x]+[x, d(g)]=\textbf{ad}_{d(g)}(x).
\end{multline*}

Let $d_0\in \Inn^*$, i.e. $\chi_{d_0}$ is trivial on loops. Then $ \chi_{d_0}=\sum_{g \in G}C_g \chi_{\textbf{ad}_g}$, hence $d_0$ can be expressed in the form of an arbitrary sum $d_0=\sum_{g\in G} C_g \textbf{ad}_g$ - an arbitrary sum of adjoint derivations (which, nevertheless, maps any element of $G$ to a finite sum). Then
$$[d, d_0](x)=[d,\sum_{g\in G} C_g \textbf{ad}_g](x)=\sum_{g\in G}C_g  [d,\textbf{ad}_g](x)=\sum_{g\in G}\textbf{ad}_{d(g)}(x).$$
We know that for $\forall g$, $\textbf{ad}_{d(g)}$ has a trivial-on-loops character. Hence  $[d, d_0]$ is trivial on loops.
\end{proof}

\subsection{Characters on subgroupoids}
\begin{definition}
Let $\chi\in X_1(\Gamma)$. We define support of character $\chi $ as
$$supp(\chi)=\{\varphi\in \Hom(\Gamma)|\chi(\varphi)\neq 0\}$$
\end{definition}
Now we can look at character submodules consisting only of characters supported on some particular subgroupoid.
\begin{definition}
\begin{itemize}
\item $\Der_{[u]}=\{d\in \Der| supp(\chi_d)\subset \Gamma_{[u]}\}$;
\item $X_1(\Gamma_{[u]})=\{\chi\in X_1(\Gamma)|supp(\chi)\subset \Gamma_{[u]} \}$;
\item $X_1^{fin}(\Gamma_{[u]})=\{\chi\in X_1(\Gamma)|supp(\chi)\subset \Gamma_{[u]} , \chi\text{ is locally finite}\}$;
\item $X_0^{fin}(\Gamma_{[u]})=\{\chi\in X_1(\Gamma)|supp(\chi)\subset \Gamma_{[u]} , \chi\text{ is locally finite and trivial on loops}\}$.
\end{itemize}

\end{definition}
Properties of characters on subgroupoids have been studied in-depth in the article \cite{Arutyunov_2}. (In fact, there was only considered the case when $A=\C$. Nevertheless, all properties of characters on subgroupoids coincide in both cases, with their proofs totally identical to the case when $A=\C$.)
\begin{fact}
Some trivial facts:
\begin{itemize}
\item $X_1(\Gamma_{[u]})$ is a submodule in $X_1(\Gamma)$, $X_1^{fin}(\Gamma_{[u]})$ is a submodule in $X_1^{fin}(\Gamma)$. Similarly for $X_0$ and $X_0^{fin}$;
\item $\Der_{[u]}\cong X_1^{fin}(\Gamma_{[u]})$ as $A$-modules;
\item As a Lie algebra w.r.t. commutator, $\Der_{[u]}$ is not a subalgebra in $\Der$. In addition, if $X\in \Der_{[u]}, Y\in \Der_{[v]}$, then $[X,Y]$ is not always $0$;
\item Let  $G$ be \textbf{finite}. It follows that there is only finite number of objects and morphisms in $\Gamma$. Since a function over a disjoint union of a finite number of subdomains can be represented in a form of a direct sum, we have the following $A$-module isomorphisms:
\begin{enumerate}
\item  $X_1(\Gamma)=\bigoplus_{[u]\in G^G} X_1(\Gamma_{[u]})$,
\item $X_0(\Gamma)=\bigoplus_{[u]\in G^G} X_0(\Gamma_{[u]})$,
\item It follows that $\Der(A[G])\cong \bigoplus_{[u]\in G^G}\Der_{[u]}$.
\end{enumerate}
\end{itemize}

\end{fact}
%proof and counterexamples?
Other facts about derivations on subgroupoids can be found in \cite{Arutyunov_1,Arutyunov_2}. 
% Those properties cover the cardinality of characters on subgroupoids.
\subsection{Locally-finite characters complex}
Consider an $A$-module complex from the character modules
\begin{equation}\label{finses}
0  \xrightarrow[]{} X_0^{fin}(\Gamma_{[u]}) \xrightarrow[]{\partial_0} X_1^{fin}(\Gamma_{[u]}) \xrightarrow[]{\partial_1} X_1^{fin}(\Hom(u,u)) \xrightarrow[]{} 0.
\end{equation}

The first differential $\partial_0:X_0^{fin}(\Gamma_{[u]})\to X_1^{fin}(\Gamma_{[u]})$ is defined as an identical inclusion of functions on $\Gamma$. 

The second differential $\partial_1: X_1^{fin}(\Gamma_{[u]})\to X_1^{fin}(\Hom(u,u))$ is a projection on a group of loops over $u$
 $$\partial_1(\chi)(\varphi)=\begin{cases}
\chi(\varphi), &\varphi \in \Hom(u,u)\\
 0,&  else
 \end{cases}$$
 Talking simply, $\partial_1$ is just a restriction of character on a $\Hom(u,u)$.
 
 \begin{theorem}
 The sequence 
 \begin{equation}\label{ses_groupoids}
0  \xrightarrow[]{} X_0^{fin}(\Gamma_{[u]}) \xrightarrow[]{\partial_0} X_1^{fin}(\Gamma_{[u]}) \xrightarrow[]{\partial_1} X_1^{fin}(\Hom(u,u)) 
\end{equation}
is an exact sequence of $A$-modules.
 \end{theorem}
\begin{proof}

Firstly, $\partial_1\circ \partial_0=0$. Since a restriction of a trivial-on-loops character on a loop group on some particular vertex is zero.

Secondly,we will prove that the sequence \eqref{ses_groupoids} is exact:
\begin{enumerate}
\item $\ker(\partial_0)=\{0\}$, since if character is trivial on loops, and equals 0 on non-loops, then it vanishes everywhere;
\item $\im\partial_0\equiv X_0^{fin}$;
\item Let us prove that $\ker \partial_1\subset \im \partial_0$. By definition, $\chi \in \ker\partial_1 \iff \chi$ is locally-finite,  and  $\forall \varphi\in \Hom(u,u), \chi(\varphi)=0$. It follows that $\chi$ is trivial on any loop in $\Gamma_{[u]}$ due to character properties \eqref{character_properties}. Hence  $\chi\in \im\partial_0$. 
\end{enumerate}

\end{proof} 
 
Contrary to exactness of the sequence \eqref{ses_groupoids}, the sequence \eqref{finses} is not always exact. %counterexamples can be derived from results of the paper \cite{Arutyunov_1}.
%there are counterexamples for groups with conjugacy classes of infinite cardinality. In these counterexamples characters from $X_1^{fin}(\Hom(u,u))$  have non-locally-finite characters in their preimage under $\partial_1$. In homological terms, this means that in such a case homology of the sequence  \eqref{finses} does not vanish at the last term.
 %link
 
 However, we can prove that sequence of the form \eqref{finses} becomes exact in case $[u]$ has finite size.
 \subsection{Non-locally-finite characters complex}
There are similar character modules and a similar sequence of characters for non locally-finite characters. In case $A=\C$ this sequence has been studied in papers \cite{Arutyunov_1, Arutyunov_2}.
  \begin{definition}
  \begin{itemize}
      \item  $X_1(\Gamma)$ is a set of 1-characters on groupoid $\Gamma$;
      \item $X_0(\Gamma)$ is a set of trivial on loops 1-characters on groupoid $\Gamma$.
  \end{itemize}

  \end{definition}
 Then there is a (generally, non-exact) sequence of characters.
  
\begin{equation}\label{genses}
0  \xrightarrow[]{} X_0(\Gamma_{[u]}) \xrightarrow[]{\partial_0} X_1(\Gamma_{[u]}) \xrightarrow[]{\partial_1} X_1(\Hom(u,u)) \xrightarrow[]{} 0.
\end{equation}
For some particular classes of groups $G$, this sequence coincides with \eqref{ses_groupoids}. In addition, it can be described in terms of cohomology of 2-category (2-groupoid) defined similarly to $\Gamma$. Details on this approach can be found in paper \cite{Arutyunov_4} (see Theorem 2.1 and Proposal 4.4).
%link

 \subsection{Characters on subgroupoid corresponding to finite conjugacy class}
 
In this section we study the $A$-module of characters over $\Gamma_{[u]}$, when the conjugacy class $[u]$ is finite. This provides some techinques, which are useful for studying $\Der(A[G])$ in case the group $G$ is a finite group or an FC-group. (Detailed description of FC-groups and their properties can be found in \cite{Robinson_book}). 

In addition, these techniques provide a way to solve a Derivation problem (i.e. find out whether $\Out(A[G])=0$ ) in some cases when $G$ \textit{has} at least one finite conjugacy class. A special example is the case when $G$ is an FC-group, i.e. all conjugacy classes of $G$ are finite.
\begin{lemma}\label{finiteness_intersection}
Let $g\in G$, and let $|u|<\infty$. Then there are only finitely many morphisms of the form $(*,g)$ in $\Hom(\Gamma_{[u]})$. 
\end{lemma}
\begin{proof}
For any $g\in G$,
\begin{multline*}
|h:(h,g)\in \Hom(\Gamma_{[u]})|=
|\{h:g^{-1}h,hg^{-1}\in [u]\}|=\\
(\text{since }hg^{-1}\in [u]\iff g^{-1}hg^{-1}g=g^{-1}h\in [u])
\\=|\{h:g^{-1}h\in [u]\}|=|\{h:h\in g[u]\}|
=|g[u]|=|[u]|<\infty.
\end{multline*}
\end{proof}
In terms of paper \cite{Arutyunov_3} (section \textbf{2.3}, p.7), this means $|[u]|<\infty\implies|\cH_g\cap \Gamma_{[u]} |<\infty$.
\begin{corollary}
If $|u|<\infty$, then any character with support in $\Gamma_{[u]}$ is automatically locally finite, i.e. $ X_1^{fin}(\Gamma_{[u]})\equiv X_1(\Gamma_{[u]})$, $X_0^{fin}(\Gamma_{[u]})\equiv X_0(\Gamma_{[u]})$.
\end{corollary}
\begin{proof}
Let a character $\chi$ be supported in $\Gamma_{[u]}$. Then for any $g$ we know that  all morphisms of the form $(*,g)$, such that $\chi(*,g)\neq 0$, lie in in $\Hom(\Gamma_{[u]})$. Due to the lemma above, there can only be a finite number of such morphisms in $\Hom(\Gamma_{[u]})$. Thus, $\chi$ is locally finite.
\end{proof}
The following lemma is an important technical tool in this paper. Its consequences will be used later to prove decomposition theorems for derivations of finitely generated FC-group rings and finite group rings.
\begin{lemma}\label{main}
Let the conjugacy class $[u]$ be finite. Then the sequence 
 \begin{equation}
0  \xrightarrow[]{} X_0^{fin}(\Gamma_{[u]}) \xrightarrow[]{\partial_0} X_1^{fin}(\Gamma_{[u]}) \xrightarrow[]{\partial_1} X_1^{fin}(\Hom(u,u))  \xrightarrow[]{} 0
\end{equation}
is a split exact sequence of $A$-modules.
 \end{lemma}
 \begin{proof}
 $ $\newline
1. \textit{Exactness part}

For exactness, we need only $\im\partial_1=X_1^{fin}(\Hom(u,u)) $. 

So, let $\chi_u\in X_1^{fin}(\Hom(u,u))$. Then we can define $\chi\in \partial_1^{-1}(\chi_u)$ in the following way:\begin{equation}\label{section}
\chi(\varphi)=
\begin{cases}
\chi_u(\varphi) & \text{ if } \varphi\in \Hom(u,u),\\
0  & \text{ if } \varphi \text{ is not a loop in }\Gamma[u],\\
\chi_u(\theta^{-1}\circ\varphi\circ\theta) & \text{ if }\varphi\in \Hom(u^\prime,u^\prime).
\end{cases}
\end{equation}

In the last case, $\theta$  is any morphism in $\Hom( u^\prime, u)$. Obviously, due to characters' properties \eqref{character_properties}, $\chi(\varphi)$ does not depend on the choice of $\theta$. In addition,
\begin{multline*}
\chi(\varphi\psi)=\chi_u(\theta^{-1}\varphi\psi\theta)=\chi_u(\theta^{-1}\varphi\theta\theta^{-1}\psi\theta)=\\
=\chi_u(\theta^{-1}\varphi\theta)+\chi_u(\theta^{-1}\psi\theta)=
\chi(\varphi)+\chi(\psi)
\end{multline*} This means that character $\chi$ is defined correctly.

\textit{We have to prove that $\chi$ is locally finite (on $\Gamma_{[u]}$).}

Let $g\in G$, then there must be finite number of $h$'s, such that $(h,g)\in \Hom(\Gamma_{[u]})$ and $ \chi((h,g))\neq 0$. 

On one hand, $(h,g)\in \Hom(\Gamma_{[u]})\implies g^{-1}h, hg^{-1}\in [u]$. 

On the other hand $\chi((h,g))\neq 0\implies (h,g)$ is a loop $\implies g^{-1}h=hg^{-1}$. We get that
$$|\{h:h\in Z(g),hg^{-1}\in [u]\}|=|\{h:h\in Z(g)\cap [u]g\}|\leq |[u]g|<\infty$$
hence $\chi$ is locally-finite.

2. \textit{Split part}

Denote  $ id_{X_1^{fin}(\Hom(u,u)) }$ as the identity map on $X_1^{fin}(\Hom(u,u))$. The definition above defines a section for $\partial_1$, i.e. the map $f:\chi_u\mapsto\chi$ , has the property: $\partial_1\circ f= id_{X_1^{fin}(\Hom(u,u)) }$. Hence  the sequence \eqref{ses_groupoids} splits.
\end{proof} 
\begin{definition}
Denote the map described above \eqref{section} as $f: \chi_u\mapsto \chi$. This map will be used later
\end{definition}
\begin{corollary}\label{decomp_subgroupoid}
It follows that $\Der_{[u]}$ is isomorphic to
$X_0^{fin}(\Gamma_{[u]}) \oplus X_1^{fin}(\Hom(u,u))$.
\end{corollary}
\begin{corollary}
If $|u|<\infty$, then
 \begin{equation}
0  \xrightarrow[]{} X_0(\Gamma_{[u]}) \xrightarrow[]{\partial_0} X_1(\Gamma_{[u]}) \xrightarrow[]{\partial_1} X_1(\Hom(u,u))  \xrightarrow[]{} 0
\end{equation}
is a split exact sequence of $A$-modules. Here the differentials are same as for the sequence above.
\end{corollary}

% \begin{proof}
% Due to the lemma above, any character with support in $\Gamma_{[u]}$ is automatically locally finite. hence $ X_1^{fin}(\Gamma_{[u]})\equiv X_1(\Gamma_{[u]})$, $X_0^{fin}(\Gamma_{[u]})\equiv X_0(\Gamma_{[u]})$.
% \end{proof}

% Following theorems provide tools to explicitly describe derivations of group rings. Unfortunately, the second one works only in case $A$ is a Dedekind domain.
\begin{definition}
\label{def-Hom_ab}
Let $H$ be some group, $A$ be a ring. We denote $\Hom_{Ab}(H,A)$ as the set of additive homomorphisms from the group $H$ to the ring $A$. 
\end{definition}
\begin{lemma}\label{loop_group}
The following isomorphism holds
$$X_1(\Hom(u,u))\cong \Hom_{Ab}(Z(u), A).$$
\end{lemma}

\begin{proof}
It is easy to notice that $\Hom(u,u)$ is a group with respect to morphism composition. Obviously, this group is canonically isomorphic to centralizer subgroup $Z(u)$. In other words, characters on $\Hom(u,u)$ are nothing but additive homomorphisms from $Z(u)$ to $A$. 
\end{proof}

\section{Derivations of FC-group rings over finitely generated FC groups}
 Recall that a group $G$ is called an FC-group (see \cite{Robinson_book}), if all conjugacy classes of $G$ are of a finite size. 
 \begin{lemma}\label{FC1}
Let $G$ be a finitely generated FC-group, $\Gamma$ be its groupoid of adjoint action. Then $X_1^{fin}$ can be decomposed into a direct sum of submodules

 \begin{equation}
X_1^{fin}(\Gamma)=\bigoplus_{[u]\in G^G} X_1^{fin}(\Gamma_{[u]}).
 \end{equation}

 \end{lemma}
 We identify character supported on $\Gamma_{[u]}$ with its extension on the whole $\Gamma$, such that the extended character has zero value on any morphisms outside of $\Gamma_{[u]}$.
 \begin{proof}
1) The inclusion $ \bigoplus_{[u]\in G^G} X_1^{fin}(\Gamma_{[u]})\subset X_1^{fin}(\Gamma)$ is obvious.
 
 In fact, if $ \chi\in \bigoplus_{[u]\in G^G} X_1^{fin}(\Gamma_{[u]})$, then by definition of infinite direct sum of modules, $\chi$ is a finite sum of locally-finite characters. This implies that $\chi$ is locally-finite.
 
2) Let us prove that $ X_1^{fin}(\Gamma)\subset\bigoplus_{[u]\in G^G} X_1^{fin}(\Gamma_{[u]}) $. This is actually equivalent to the following statement: if $\chi$ is locally finite, then $supp(\chi)$ lies in a finite union of subgroupoids $\bigsqcup_{i=1}^m \Gamma_{[u_i]} $ for some $m$. So, let us prove this.

By assumption, $G$ is finitely generated. Let us choose and fix its presentation: $$G=\langle s_1, \ldots, s_n| R \rangle\equiv \langle S| R \rangle.$$This means that any element $v$ can be (perhaps non-uniquely) represented in the form (see \cite{lyndon-schupp-1980}):  $$v=a_1\ldots a_k,$$
where $a_1,\ldots, a_k$ are generating elements of $G$ or their inverses
$$a_1=s_{i_1}^{\pm 1},\ldots,a_k= s_{i_k}^{\pm 1}.$$

In terms of groupoid of adjoint action, it means that any morphism $  (u,v)=(u,a_1\ldots a_k)\in \Hom(\Gamma)$ can be decomposed into a composition of morphisms in the following way
\begin{multline}\label{decomp}
(u,a_1\ldots a_k)=(u , a_1)\circ (a_1^{-1}u , a_2)\circ\ldots\circ\\
\circ(a_{k-2}^{-1}\ldots a_1^{-1} u, a_{k-1})\circ(a_{k-1}^{-1}\ldots a_1^{-1} u, a_k).
\end{multline}
 Now consider a locally-finite character $\chi\in X_1^{fin}(\Gamma)$. Since $\chi$ is locally-finite, then for any generating element $s_i$ there is only a finite number of morphisms of the form $(h,s_i)$, such that $\chi$ does not vanish on these morphisms (similarly to \cite{Arutyunov_1}). Hence  there can only be a finite number of subgroupoids $\Gamma_{[u_i]}$ containing morphisms of the form $(h,s_i)$, such that $\chi((h,s_i))\neq 0$. 
 
 Similarly for any generating element's inverse $s_i^{-1}$.

Since $G$ is finitely generated, all morphisms $(*,g) $ such that $g\in S\sqcup S^{-1}$ and $\chi((*,g))\neq 0$ are contained in a finite union of  a finite unions of subgroupoids $  \Gamma_{[u_i]}$.  This implies that all such morphisms are contained in a finite union of subgroupoids, and without any loss of generality we can denote it as $\bigsqcup_{i=1}^m \Gamma_{[u_i]} $.

Now recall that any morphism $(u,v)$ can be decomposed in the form \eqref{decomp}. We conclude that if $\chi((u,v))\neq 0$, then $(u,v)\in \Hom(\bigsqcup_{i=1}^m \Gamma_{[u_i]}) $. Hence  support of $\chi$ lies in $\bigsqcup_{i=1}^m \Gamma_{[u_i]}$.

\end{proof}
\begin{corollary}\label{FC2}
 The following decomposition holds
 \begin{equation}
X_0^{fin}(\Gamma)=\bigoplus_{[u]\in G^G} X_0^{fin}(\Gamma_{[u]}).
 \end{equation}
\end{corollary}
\begin{proof}
The proof for this statement is completely the same as for the statement above.
\end{proof}
\begin{theorem}\label{FC3}
If $G$ is a finitely generated FC-group, all weakly-inner derivations of $A[G]$ are inner.
\end{theorem}
In order to prove this theorem, we need the following lemma.
\begin{lemma}
If $|[u]|<\infty$, then all weakly inner derivations whose characters are supported on $\Gamma_{[u]}$ are inner derivations.
\end{lemma}

\begin{proof}
%\textit{(lemma)}

If we apply lemma \eqref{weakly_inner}, we know that any trivial-on-loops character $\chi$ on  $\Gamma_{[u]}$ can be expressed as a (possibly infinite) sum of adjoint derivation characters
$$\chi=\sum_{g\in G}C_g \chi_{\textbf{ad}_g}.$$

% Obviously, if $C_g\neq 0$, then $supp(\chi_{\textbf{ad}_g})\subset supp(\chi)$
However, if $supp(C_g\chi_{\textbf{ad}_g})\subset \Gamma_{[u]}$, then $g\in \Ob(\Gamma_{[u]})=[u]$. It follows that
$$\chi=\sum_{g\in [u]}C_g \chi_{\textbf{ad}_g}.$$
Since $[u]$ is finite, $\chi$ is a finite sum of adjoints. Hence  it is a character of an inner derivation.
\end{proof}

\begin{proof}
\textit{(Theorem\eqref{FC3})}

Let $d$ be a weakly inner derivation. Then its character $\chi_d$ is locally-finite and trivial on loops. Due to previous lemma \eqref{FC2}, $\chi_d$ lies in a direct sum $X_0^{fin}(\Gamma_{[u]})$. Therefore, it is a finite sum of characters, whose supports lie in particular subgroupoids. It follows that $\chi_d$ is a finite sum of characters, each of those corresponding to a finite sum of adjoints. 
Thus, $\chi_d$ is a finite sum of finite sums of adjoint derivations' characters, and its corresponding derivation $d$ is a finite sum of adjoints. This means that $d$ is an inner derivation.

\end{proof}
Now we can derive the structure theorem for derivations of finite group rings over finitely generated FC-groups.

\begin{theorem}
% Main result.
\label{th-FC-decomp}
Let $G$ be a finitely generated FC-group, $A$ be a unital commutative ring. Then 
\begin{equation}
\Der(A[G])\cong \Inn(A[G])\oplus \bigoplus_{[u]\in G^G} \Hom_{Ab}(Z(u),A).
\end{equation}
\end{theorem}
\begin{proof}

As $A$-modules (not as Lie algebras):
$$\Der(A[G])\cong X_1^{fin}(\Gamma).$$

Following the lemma \eqref{FC1},
$$X_1^{fin}(\Gamma)=\bigoplus_{[u]\in G^G} X_1^{fin}(\Gamma_{[u]})$$
Following the theorem \eqref{main}, we can combine the decomposition of character modules on subgroupoids \eqref{decomp_subgroupoid} 
$$X_1^{fin}(\Gamma_{[u]})\cong X_0^{fin}(\Gamma_{[u]}) \oplus X_1^{fin}(\Hom(u,u))$$
From the corollary of lemma \eqref{finiteness_intersection}, we know that all characters on $\Gamma_{[u]}$ are locally-finite. In addition, we know that there is an explicit desription of characters on loop hom-sets \eqref{loop_group}. Therefore, we have:
$$X_1^{fin}(\Gamma_{[u]})\cong X_0^{fin}(\Gamma_{[u]}) \oplus \Hom_{Ab}(Z(u), A)$$

Now recall that we have an $A$-module isomorphism between $\Inn^*$ and $X_0^{fin}$. If we combine it with corollary \eqref{FC2}, we obtain the isomorphism:
$$\Inn^*\cong X_0^{fin}= \bigoplus_{[u]\in G^G} X_0^{fin}(\Gamma_{[u]}).$$
Now recall that due to the theorem \eqref{FC3}, we have $\Inn^*= \Inn$. Regrouping terms in the equation above, we obtain the theorem claim:
\begin{multline}
\Der(A[G])\cong X_1^{fin}(\Gamma)=\bigoplus_{[u]\in G^G} X_1^{fin}(\Gamma_{[u]})\cong \\ 
\cong \bigoplus_{[u]\in G^G} \big(X_0^{fin}(\Gamma_{[u]}) \oplus \Hom_{Ab}(Z(u), A)\big) 
\cong\Inn^*(A[G])\bigoplus_{[u]\in G^G} \Hom_{Ab}(Z(u), A)\cong \\ \cong \Inn(A[G])\bigoplus_{[u]\in G^G} \Hom_{Ab}(Z(u), A).
\end{multline}

\end{proof}

In the next section, we will be interested in the case of finite groups. So the following corollary will be important.

 \begin{corollary}\label{corol-finitegroup}
In case $G$ is finite, $\Inn^*(A[G])=\Inn(A[G])$.
\end{corollary}

\begin{proof}
This is a special case of theorem \eqref{FC3}.
\end{proof}

\section{Finite group case}
In this section group $G$ is assumed to be finite. In this case we see that its groupoid of adjoint action $\Gamma$ has  both finite number of objects and morphisms. Thus, any character has finite support, so any character on $\Gamma$ is a locally-finite character.

Thus, taking in account the corollary \ref{corol-finitegroup} we get
\begin{proposition}\label{prop-finite-outer}
$\Out^*(A[G])\cong \Out(A[G])$.
\end{proposition}

%\subsection{Structure theorem for $\Der(A[G])$ for finite $G$ and arbitrary commutative unital ring $A$}

Let us refine the theorem \ref{th-FC-decomp} for the case of finite group.

\begin{theorem}\label{finite_result}
There is a way to describe all derivations of $A[G]$ in case $G$ is finite:
\begin{enumerate}
\item $\Der(A[G])\cong\Inn(A[G]) \oplus \big(\bigoplus_{[u]\in G^G}   \Hom_{Ab}(Z(u), A)\big),$
\item $\Out(A[G])\cong \bigoplus_{[u]\in G^G}   \Hom_{Ab}(Z(u), A).$
\end{enumerate}
Notice that these isomorphisms are isomorphims of $A$-modules, and they are not canonical.
\end{theorem}
\begin{proof}
1. As $A$-modules (not as Lie algebras):
$$\Der\cong X_1^{fin}(\Gamma).$$

Due to finiteness of $G$:
$$ X_1^{fin}(\Gamma)= X_1(\Gamma).$$
As an $A$-module of $A$-valued functions over a finite disjoint union of domains, $X_1(\Gamma)$ can be decomposed:

$$ X_1(\Gamma)\cong\bigoplus_{[u]\in G^G} X_1(\Gamma_{[u]}).$$
Now, due to the theorem \eqref{main}
\begin{multline*}
X_1(\Gamma_{[u]})= X_1^{fin}(\Gamma_{[u]})\cong X_0^{fin}(\Gamma_{[u]}) \oplus X_1^{fin}(\Hom(u,u))\equiv\\ 
\equiv X_0(\Gamma_{[u]}) \oplus X_1(\Hom(u,u)).
\end{multline*}
Hence  
$$X_1(\Gamma)\cong\bigoplus_{[u]\in G^G} \big(X_0(\Gamma_{[u]}) \oplus X_1(\Hom(u,u))\big).$$
On one hand, we know that 
$$\bigoplus_{[u]\in G^G}X_0(\Gamma_{[u]})= X_0(\Gamma)\cong \Inn.$$
On the other hand,  $\Hom(u,u)$ is a group with respect to morphism composition, and this group is isomorphic to centralizer $Z(u)$. In other words, characters on finite$\Hom(u,u)$ are nothing but additive homomorphisms from $Z(u)$ to $A$. I.e. we can delineate the following $A$-module homomorphism
 $$X_1(\Hom(u,u))\cong\Hom_{Ab}(Z(u), A).$$
Regrouping terms in the equation above, we obtain that
$$\Der\cong X_1(\Gamma)\cong\Inn \oplus \bigoplus_{[u]\in G^G}   \Hom_{Ab}(Z(u), A).$$
2. Trivially follows from the first part and proposition \ref{prop-finite-outer}.
\end{proof}

\subsection{Inner derivations}
If the group $G$ is finite,  we can suppose that the dimension of the space of inner derivations is the difference between number of generators $ad_g, g\in G$ and  the dimension of centralizer of a group algebra $A[G]$, which is equal to the number of conjugacy classes $|G^G|$ ( \href{https://ysharifi.wordpress.com/2011/02/08/center-of-group-algebras/}{here} is the proof for the case $A=\C$). In this paragraph we prove this formula for arbitrary $A$:
$$dim(\Inn(A[G]))=|G|-|G^G|.$$In order to prove this, we need the following well-known fact (see \cite{Brakenhof}, chapter 2).

For a fixed conjugacy class $[u_i]\in G^G$, denote $K_i:=\sum_{g\in [u_i]}g$.
\begin{fact}\label{center}
%On the center of a group ring
An element $x$ lies in the center of $A[G]$ if and only if $x=C_1 K_1+\ldots + C_n K_n$, for some elements $C_i\in A$.
\end{fact}

\begin{theorem}
Inner derivations can be described in terms of an $A$-module presentation
$$\Inn=\langle \textbf{ad}_{g}, g\in G |\sum_{g\in [u]}\textbf{ad}_{g}=0 \text{ for any }[u]\in G^G \rangle$$
\end{theorem}
Details on module presentations can be found at \cite{Herrmanns}. This theorem is equivalent to the following
\begin{corollary}\label{inner_derivations}
%\textbf{(basis in inner derivations)}

Let us choose some $u_1, \ldots, u_n$, such that all the $u_i$'s belong to different conjugacy classes.  Then  $\Inn(A[G])$  is a free $A$-module, with the following basis: $$\big\{\textbf{ad}_{g}|g\in G\setminus \{u_1, \ldots, u_n\}\big\}.$$

\end{corollary}

So let's prove the theorem with the corollary.
\begin{proof}
1)By the definition of $\Inn$, we know that $ \textbf{ad}_{g}, g\in G $ is a generating set for $\Inn$. In addition, we know that $ \textbf{ad}_{x}=0 \iff x$ lies in the center of the group algebra $A[G]$. Combining this with the description of the center of a group algebra, we obtain the claim of the theorem.

In order to prove the corollary, it suffices to prove the rest for any conjugacy class. 

So, we will fix a conjugacy class $[u]=\{g_1,\ldots, g_k\}$. Consider inner derivations, such that their characters are supported in $\Gamma_{[u]}$.  In addition, let $$m\in\{1,\ldots,k\}.$$

2) Firstly, we have to prove that $\{\textbf{ad}_{g}|g\in [u]\setminus g_m\}$ is truly a generating set for inner derivations supported on $\Gamma_{[u]}$. 

Denote $$\Inn(\Gamma_{[u]}):=\{d\in\Inn(A[G])|supp(\chi_d)\subset\Gamma_{[u]}\}.$$

Any inner derivation $d$ supported in $\Gamma_{[u]}$ is a linear combination of adjoints $\textbf{ad}_g$, such that $g\in [u]$. In addition, due to the lemma \eqref{center},$$\sum_{i } g_i\in C_{A[G]}(A[G])\implies 0=\textbf{ad}_{\sum_{i } g_i}=\sum_{i }\textbf{ad}_{g_i}.$$
Thus, $\textbf{ad}_{g_m}$ is a linear combination of others:

$$\textbf{ad}_{g_m}=\sum_{i, i\neq m} -\textbf{ad}_{g_i}.$$

3)Let $$\sum_{i, i\neq m} C_i \textbf{ad}_{g_i}=0$$
for some non-zero coefficients $C_i\in A$. Then for any $x\in A[G]$, we have 
\begin{equation}
0=(\sum_{i, i\neq m} C_i \textbf{ad}_{g_i})(x)=\sum_{i, i\neq m} C_i \textbf{ad}_{g_i}(x)=\sum_{i, i\neq m} C_i [g_i,x]
= [\sum_{i, i\neq m} C_i g_i,x].
% =\textbf{ad}_{\sum_i C_i g_i}
\end{equation}
This means that this linear combination must lie in the group ring's center:
$$\sum_{i, i\neq m} C_i g_i\in C_{A[G]}(A[G]).$$ 
We see that this contradicts \eqref{center}. Therefore, all $C_i$'s must vanish, and $$
\{\textbf{ad}_{g}|g\in [u]\setminus g_m\}$$
must be linearly independent.

4) Due to the lemma above(\eqref{FC2}), $$\Inn(A[G])= \bigoplus_{[u]\in G^G} \Inn(\Gamma_{[u]}).$$As a consequence, $\Inn(A[G])$ is a free finitely generated module, spanned by
\begin{multline}
\bigsqcup_i\{\textbf{ad}_{g}|g\in [u_i]\setminus u_i\}=\\
=\{\textbf{ad}_{g}|g\in G\setminus \{u_1,\ldots, u_n|u_i\text{'s belong to different conjugacy classes }\}\}.
\end{multline}

\end{proof}
\begin{corollary}
If $A$ is a field, then $$dim(\Inn)=\sum_{i=1}^n(|[u_i]|-1)=|G|-|G^G|.$$
\end{corollary}
\begin{proof}
 It's a widely known fact that a free finitely generated module over a field is a vector space. Its dimension is equal to the number of basis elements, which we have calculated in the theorem \eqref{inner_derivations}.
\end{proof}

\subsection{Characters of outer derivations and torsion}
Now we will describe characters corresponding to the loop groups. This gives us criterion to find out whether all derivations are inner.
\begin{lemma}\label{char_condition}
Let $H$ be any subgroup of $G$, $A$ be a ring.
If $\varphi\in \Hom_{Ab}(H,A)$, then  for any $ \forall g\in H$ we have $ ord(g)\varphi (g)=0 \in A$  
% If $\varphi\in \Hom_{Ab}(Z(u),A)$, then  $ \forall g\in Z(u), ord(g)\varphi (g)=0 \in A$ 
% and $\varphi\neq 0$, 
\end{lemma}
\begin{proof}

 Firstly, if $ord(g)=k$ in $G$, then $g$ has the same order in $H$.

Secondly, since $|G|<\infty$, $\forall g\in H$,  we have  $ord( g) <\infty$. Then $0=\varphi (e)=\varphi (g^{ord(g)})=ord(g)\varphi (g).$

\end{proof}

\begin{corollary}
Let $G$ be a  finite group, $A$ be a unital commutative ring.

Then $\Der(A[G])\neq \Inn(A[G])$ if and only if there exists $ [u]\in G^G$ and a nontrivial homomorphism $\varphi \in\Hom_{Ab}(Z(u),A)$ such that $\forall g\in Z(u), ord(g)\varphi (g)=0 \in A$.

\end{corollary}

\begin{corollary}\label{corol-torsion-free}
If $A$ is a torsion-free ring, $G$ is a finite group, then $\Out(A[G])=0$. As a consequence, if $A=\Z,\Q,\R,\C$, then $\Out(A[G])=0$.

\end{corollary}

Hereinafter we use notation $gcd(a,b)$ for the greatest common divisor.

\begin{lemma}\label{th-outer-trivial-divisor}
Let $A=\Z_m$, $G$ - finite. Then $$\forall g\in G, gcd(ord(g), m)=1\implies \Der(A[G])=\Inn(A[G]).$$

\end{lemma}
\begin{proof}

Let $\varphi\in \Hom_{Ab}(Z(u),Z_m)$ for some $u$. Then for any $g\in Z(u)$ we have $$ord(g)\varphi(g)=0\in Z_m\implies m|ord(g)\varphi(g).$$Since $\forall g\in G, gcd(ord(g), m)=1$, we have $$m|ord(g)\varphi(g)\implies m|\varphi(g)\implies \varphi(g)=0\in Z_m.$$
\end{proof}
\subsection{Additive homomorphisms}
The following fact is widely-known amongst group theorists.
\begin{fact}
Let $\varphi:G\to A$ be an additive homomorphism from a non-abelian group $G$ to abelian group $A$. Let us denote $\pi:G\to G/ [G,G]$ as the projection mapping. Then there exists $\hat{\varphi}$, such that $\varphi=\hat{\varphi}\circ \pi$.
\begin{center}
% https://tikzcd.yichuanshen.de/#N4Igdg9gJgpgziAXAbVABwnAlgFyxMJZABgBpiBdUkANwEMAbAVxiRAHEQBfU9TXfIRQAmclVqMWbAILdeIDNjwEiARlKrx9Zq0QcA9Mnal2FbuJhQA5vCKgAZgCcIAWyRkQOCEnUgGdACMYBgAFfmUhPxh7HBBqbSk9AB0k+kc0AAssOQdnN0QPLyRRCR02FLRsnlzXYuoixF8E3RAUjLocYBS0zKwuOL8sMBaoOjgMy3MuIA
\begin{tikzcd}
G \arrow[rr, "\varphi"] \arrow[rd, "\pi"] &                                               & A \\
                                          & {G/[G,G]} \arrow[ru, "\hat{\varphi}", dashed] &  
\end{tikzcd}
\end{center}
\end{fact}
If we use the primary decomposition of finite abelian groups, then we can decompose the additive group of ring $A$ 
$$A\cong\Z_{p_1^{i_1}}\oplus ... \oplus\Z_{p_n^{i_n}},$$
and$$G/[G,G]\cong\Z_{q_1^{j_1}}\oplus ... \oplus\Z_{q_m^{j_m}}.$$
\begin{lemma}
There is a nontrivial homomorphism $\varphi: \Z_{p^{i}}\to \Z_{q^{j}}$ if and only if $p=q$.
\end{lemma}
\begin{proof}
1. If $p=q$, then we can always take $\varphi:1\mapsto q^{j-1}$.

2. Let  $\varphi: \Z_{p^{i}}\to \Z_{q^{j}}$  be a homomorphism. Notice that elements of the primary group $\Z_{p^i}$  can only be of order $1, p,p^2,...,p^i$ due to the Lagrange theorem (order of element should divide order of group). Similarly we can say that elements of $\Z_{q^j}$ can only be of order $1,q,...,q^j$. Then for any element $g\in \Z_{p^i}$ of order $p^k$ we have  $0=\varphi(0)=\varphi(p^k g)=p^k\varphi( g)$.  On the other hand we know that $ord(\varphi(g))\in 1,q,...,q^j$. It can only be possible if $p=q$.
\end{proof}
\begin{corollary}
\label{corol-Hom-ab}
Let the primary decompositions of  $A$ and $G/[G,G]$ be 
$$A\cong\Z_{p_1^{i_1}}\oplus ... \oplus\Z_{p_n^{i_n}},G/[G,G]\cong\Z_{q_1^{j_1}}\oplus ... \oplus\Z_{q_m^{j_m}}.$$Then there is a nontrivial homomorphism in $\varphi\in \Hom_{Ab}(G/[G,G],A)$ if and only if there is some $p_i,q_j$ such that $p_i=q_j$.
\end{corollary}
Now we can obtain the criterion for the triviality of the outer derivations algebra.

Let the ring $A$ be such that
$$A\cong\Z_{p_1^{i_1}}\oplus ... \oplus\Z_{p_n^{i_n}}$$
is the primary decomposition for additive group of the ring $A$. Also let 
$$G/[G,G]\cong\Z_{q_1^{j_1}}\oplus ... \oplus\Z_{q_m^{j_m}}$$be the primary decomposition of $G/[G,G]$.

\begin{theorem}
\label{th-deriv-problem}
For finite group $G$ and finite ring $A$ all derivations $\Der(A[G])$ are inner ($\Der(A[G])=\Inn(A[G])$ )
if and only if  $\{p_1,...,p_n\}\cap\{q_1,...,q_m\}=\emptyset$. 
\end{theorem}
\begin{proof}

1)Let $\{p_1,...,p_n\}\cap\{q_1,...,q_m\}=\emptyset$. Due to corollary \eqref{corol-Hom-ab} we know that this implies $\Hom(G/[G,G], A)=0$. We know that for any homomorphism $\varphi:G\to A$ there is $\hat{\varphi}\in \Hom(G/[G,G], A)$ such that $\varphi=\hat{\varphi}\circ \pi$.  Thus $\varphi=0\circ \pi$ must be the zero homomorphism.

2)
Let $$\{p_1,...,p_n\}\cap\{q_1,...,q_m\}=r.$$ 
Without any loss of generality we can say that $r=p_k=q_k$. If it's not we can just reorder primary components. Then we have a nontrivial homomorphism $$\varphi_k=\Z_{r^{i_k}}\to \Z_{r^{j_k}},\varphi_k:1\mapsto r^{j_k-1}$$
between primary subgroups  $\Z_{r^{i_k}}\subset G/[G,G]$ and $\Z_{r^{j_k}}\subset A$ correspondingly. 

Firstly, we can extend $\varphi_k$ onto the whole $G/[G,G]$ by zeros on all primary components except for $\Z_{r^{i_k}}$. Thus we have obtained $\hat{\varphi}: G/[G,G]\to A$.

Secondly, we can just compose $\hat{\varphi}$ with the projection map $\pi$:
$$\varphi=\hat{\varphi}\circ \pi.$$Now we can clearly see that there is a nontrivial central derivation $d_\varphi^e$. Since central derivations are isomorphic to a certain subalgebra of $\Out(A[G])$, then in this case $\Out(A[G])\neq 0$.
\end{proof}
% Denote $e_k$ as the element of $G/[G,G]$ corresponding to the $1$ in the primary component $\Z_{r^{i_k}}$. Then $\varphi$ takes form
% $$\varphi( h):=r^{j_k-1}, \text{ for any }h\in \pi^{-1}(e_k),$$
% $$\varphi( h):=r^{2(j_k-1)}, \text{ for any }h\in \pi^{-1}(2 e_k),$$

% Let $g^\prime\in G/[G,G]$, then define $\varphi(g):= \hat{\varphi}(g^\prime)$ for any $g\in \pi^{-1}(g^\prime)$. This is a correctly defined homomorphism: if $g,h\in G$, then we have

% $$\varphi(gh)=\hat{\varphi}(\pi(gh))=\hat{\varphi}(\pi(g)\pi(h))=\hat{\varphi}(\pi(g))+\hat{\varphi}(\pi(h))$$
\section{Examples and applications}
\label{ex-app}

\subsection{Explicit isomorphism map for outer derivations}

In the previous paragraphs we built two isomorphisms of $A$-modules:  "outer derivations" $\leftrightarrow$ "characters on loops" and "characters on loops" $\leftrightarrow$ "additive homomorphisms on centralizers". Now we will provide the explicit description of the composition of these isomorphisms. Problem is that such isomorphism is not canonical, because it depends on choice of elements in conjugacy classes. In this paragraph we build one such $A$-module isomorphism explicitly.

Fix conjugacy classes of $G$ as $$G^G=\{[u_1],[u_2],\ldots, [u_n]\}.$$
Choose representatives of conjugacy classes: $u_1\in [u_1], \ldots, u_n\in [u_n].$

Denote $$F_{u_1, \ldots, u_n}:\bigoplus_{i=1}^n   \Hom_{Ab}(Z(u_i), A)\xrightarrow{\sim} \Out(A[G])
$$
as the $A$-module isomorphism we built above. Let $$\varphi\in \bigoplus_{[u]\in G^G}   \Hom_{Ab}(Z(u), A).$$

\begin{theorem}\label{explicit_map}
Let $\varphi\in \Hom_{Ab}(Z(u_i), A)$ for some $i$. The $A$-module isomorphism mapping loop groups to outer derivations is described by the following formula:
$$F_{u_1, \ldots, u_n}(\varphi)=D(\varphi)+\Inn$$
\begin{multline*}
(D(\varphi))_v^u=
\begin{cases}
\varphi(gvg^{-1})&\text{ if }v^{-1}u=uv^{-1}=g^{-1}u_i g \text{ for some }g\in G\\
0,& \text{else }
\end{cases}.
\end{multline*}
The isomorphism $F_{u_1, \ldots, u_n}(\varphi))$ is then extended on $\bigoplus_{[u]\in G^G}   \Hom_{Ab}(Z(u), A)$ by linearity.
\end{theorem}

\begin{proof}
If $\varphi\in \bigoplus_{[u]\in G^G}   \Hom_{Ab}(Z(u), A)$, then $\varphi$ is a direct sum of homomorphisms: 
$$\varphi=( \varphi_1,\ldots ,\varphi_n),\varphi_i\in \Hom_{Ab}(Z(u_i), A).$$

We have an outer derivation $F_{u_1, \ldots, u_n}(\varphi)$ as an outer derivation, and we wish to calculate its coefficients. Then 
\begin{multline}\label{coeffouter}
(F_{u_1, \ldots, u_n}(\varphi))_v^u=(\sum_i^n \chi_i)((u,v))=\\
\text{ assume that }(u,v)\in \Gamma_{[u_i]}
\\= \chi_i((u,v))=
\begin{cases}
\chi_{\varphi_i}(\psi),& \psi\text{ is a loop, conjugated to }(u,v)\\
0,& \text{else }
\end{cases}
\end{multline}

If $\psi$ is a loop conjugated to $(u,v)$, then we have $\chi_{\varphi_i}(\psi)=\chi_{\varphi_i}((u,v))$. Since $(u,v)$ is a loop in $\Gamma_{[u_i]}$, we know that exists $g$ such that $v^{-1}u=uv^{-1}=g^{-1}u_i g$. Let us find such a $\psi$:

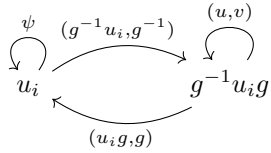
\begin{displaymath}
\begin{tikzcd} u_i 
\arrow["\psi"', loop, distance=2em, in=125, out=55]
\arrow[rr, "{(g^{-1}u_i,g^{-1})}"{anchor=south},% <-- add anchor
 bend left] &  & 
g^{-1}u_ig 
\arrow[ll, "{(u_i g,g)}"{anchor=north},% <-- add anchor
 bend left] 
\arrow["{(u,v)}"', loop, distance=2em, in=125, out=55] 
\end{tikzcd}
\end{displaymath}
Since the character $\chi_{\varphi_i}$ is potential, $\chi_{\varphi_i}(\psi)$ does not depend on the choice of $g$. Hence  we have:
\begin{multline*}
\chi_{\varphi_i}((u,v))=\chi_{\varphi_i}(\psi)=\chi_{\varphi_i}((u_ig,g)\circ(u,v)\circ(g^{-1}u_i, g^{-1}))=\\
=\chi_{\varphi_i}((u_ig,g)\circ(vg^{-1}u_i, vg^{-1}))
=\chi_{\varphi_i}((gvg^{-1}u_i, gvg^{-1}))=\\=\varphi_i(gvg^{-1}).
\end{multline*}

Combining this formula with the formula \eqref{coeffouter}, we obtain the statement of the theorem.
\end{proof}

As a result, we can describe outer derivations as an $A$-module presentation. (Details on module presentations can be found at \cite{Herrmanns}). Denote the generating set of $\Out$ as $S$, and relations as $T$. Its relations are inherited from $\bigoplus_{[u]\in G^G}   \Hom_{Ab}(Z(u), A)$  (just usual relations for addition and scalar multiplication on $\Hom$-sets with $A$-module structure). In terms of the theorem above,
\begin{multline}
\Out=<S|T>, \text{ where} \\
S=\bigsqcup_{[u]\in G^G}S_u\\
\text{Here }S_u=\{F_{u_1, \ldots, u_n}(\varphi)| \varphi\in \Hom_{Ab}(Z(u), A)\}\\
T=\{F_{u_1, \ldots, u_n}(\varphi)+F_{u_1, \ldots, u_n}(\psi)=F_{u_1, \ldots, u_n}(\varphi+\psi),  \\F_{u_1, \ldots, u_n}(a \varphi)=a F_{u_1, \ldots, u_n}(\varphi)\}\\
\text{ for any }a\in A \text{ and } \varphi,\psi \in\bigoplus_{[u]\in G^G}   \Hom_{Ab}(Z(u), A)
\end{multline}
Obviously, sets $\{S_u| [u]\in G^G\}$ are disjoint.

 \subsection{Example: derivations of $\Z_4[S_3]$}
 \label{ex-z4-s3}
Let us consider the following case: $A=\Z_4, G=S_3$. 

% Obviously, in this case, $$d\in \Inn\iff d=\sum_{g\in G}C_g \textbf{ad}_g$$

In order to describe outer derivations, we need to calculate  
$\bigoplus_{[u]\in G^G}   \Hom_{Ab}(Z(u), A)$ and then map it to $\Out(\Z_4(S_3))$  by the means of theorem \eqref{explicit_map}. Thus we can calculate explicit form of outer derivations' elements. So we begin with $\bigoplus_{[u]\in G^G}   \Hom_{Ab}(Z(u), A)$.

There are three conjugacy classes in $S_3$: $[e], [(12)],[(123)]$. 

For any permutation $\sigma$ its centralizer is defined as $Z(\sigma)=\{g\in G| \sigma=g^{-1}\sigma g\}$. Since centralizer is a subgroup, due to Lagrange's theorem on subgroup indices it must be of size 1,2,3, or 6. On the other hand, conjugation leaves the cycle type intact. Therefore,
\begin{center}
\begin{tabular}{ |c|c|c|c| } 
 \hline
\text{conjugacy class $[u]$}&[e]& [(12)]&[(123)]\\
\text{centralizer $Z(u)$ up to isomorphism}&$S_3$&$\Z_2$ &$\Z_3$\\
 \hline
\end{tabular}
\end{center}
Now we have to describe all the additive homomorphisms mapping these centralizers to $\Z_4$.
\begin{proposition}
Additive hom-sets have the following form:
\begin{enumerate}
\item $\Hom_{Ab}(S_3, \Z_4)=\{0, \varphi_1:(ij)\mapsto 2\text{ for }\forall (ij)\},$
\item $\Hom_{Ab}(\Z_2, \Z_4)=\{0, \varphi_2:1\mapsto 2\},$
\item $\Hom_{Ab}(\Z_3, \Z_4)=\{0\}.$
\end{enumerate}

\end{proposition}
\begin{proof}
The second and the third parts are obvious, because image of any element in the power of its order should be equal to zero.

Now let us prove the first part of our statement. 
Consider an arbitrary additive homomorphism $\varphi\in\Hom_{Ab}(S_3, \Z_4)$. 

For any element $g$ of $S_3$ we need $ord(g)\varphi(g)=0\in \Z_4$ to hold true. So, only $(12),(13), (23)$ can be mapped by $\varphi$ to nonzero. 

Without any loss of generality we can assume that $\varphi(12)=2$. Then 
$$2\varphi(23)+\varphi(13)=2=2\varphi(13)+\varphi(23)\implies \varphi(23)-\varphi(13)=0$$
Now if $\varphi(23)=\varphi(13)=0$, then $\varphi(123)=\varphi(23)+\varphi(12)=2$. This condradicts to the fact that $ord(123)=3$:
$$\varphi((123)^3)=\varphi(e)=0\neq 2$$

Hence  $\varphi(23)=\varphi(13)=2$, and the only nontrivial homomorphism in $\Hom_{Ab}(S_3, \Z_4)$ is the  $\varphi:(ij)\mapsto 2\text{ for }\forall (ij)$. We denote it as $\varphi_1$.

\end{proof}

% Now we can look at the table:
% \begin{center}
% \begin{tabular}{ |c|c|c|c| } 
%  \hline
% \text{conjugacy class}&[e]& [(12)]&[(123)]\\
%  $\Hom_{Ab}(Z(u), \Z_4)$& $\{0, \varphi_1:(ij)\mapsto 2\text{ for }\forall (ij)\}$&$\{0,\varphi_2:(12)\mapsto 2\}$&~---~\\
% % \text{derivation coefficients}&$d_{(12)}^{(12)}=d_{(23)}^{(23)}=d_{(13)}^{(13)}$&$d_{(12)}^{e}=d_{(23)}^{e}=d_{(13)}^{e}$&~---~\\
% % \text{their possible values}&0,2&0,2&~---~\\
% \hline
% \end{tabular}
% \end{center}

Now we can use $F_{e, (12), (123)}$  from the theorem \eqref{explicit_map} to map non-trivial homomorphisms contained in $\bigoplus_{[u]\in G^G}   \Hom_{Ab}(Z(u), A)$ to derivations.

1) We begin with the nontrivial homomorphism $\varphi_1$ from the first hom-set $\Hom_{Ab}(Z(e), \Z_4)\cong \Hom_{Ab}(S_3, \Z_4)$.
Denote the corresponding outer derivation as $$d_1+\Inn=F_{e, (12), (123)}(\varphi_1).$$
The theorem \eqref{explicit_map} gives us
\begin{multline*}
(d_1)_v^u=(F_{e, (12), (123)}(\varphi_1))_v^u=\\=
\begin{cases}
\varphi_1(gvg^{-1})&\text{ if }v^{-1}u=uv^{-1} =e\\
0,& \text{else }
\end{cases}.
\end{multline*}
So, nontrivial coefficients occur only when $v^{-1}u=uv^{-1}=e\implies u=v$,  Moreover, we see that
depending on the cycle type of $u$, by theorem \eqref{explicit_map} we have

\begin{equation*}
(d_1)_{u}^{u}=
\begin{cases}
0, &\text{ if } u=e,\\
\varphi_1(u)=2&\text{ if }u=(ij)\\
0,& \text{ if } u=(ijk).\\
\end{cases}
\end{equation*}
To conclude,
\begin{equation}\label{d1}
(d_1)_{v}^{u}=
\begin{cases}
2&\text{ if }u=v \text{ is a transposition,}\\
0,& \text{ else}\\
\end{cases}.
\end{equation}

2) The second nontrivial homomorphism is $\varphi_2\in \Hom_{Ab}(Z((12)), \Z_4)$
Denote the corresponding outer derivation as $$d_2+\Inn=F_{e, (12), (123)}(\varphi_2).$$
The theorem \eqref{explicit_map} gives us
\begin{multline}\label{jopa}
(d_2)_v^u=(F_{e, (12), (123)}(\varphi_2))_v^u=\\
=
\begin{cases}
\varphi_2(gvg^{-1})&\text{ if }v^{-1}u=uv^{-1}=g^{-1}(12) g \text{ for some }g\in G\\
0,& \text{else }
\end{cases}.
\end{multline}

On one hand,  if $v^{-1}u=uv^{-1}$, then there can be 3 cases: 
\begin{enumerate}
\item either $u=(ij),v=e$,
\item or $u=e, v=(ij)$,
\item or $u,v \text{ are cycles of length 3}$.
\end{enumerate}
Since a product of two cycles of length 3 cannot be a transposition, third case leads to a contradiction.
On the other hand, we know that for any derivation $d$, it holds true that $d_e^g=0$ for any $g$. Now we see that the first case leads to a contradiction.

It follows that only one case is possible: $u=e, v=(ij)$. Now since $g^{-1}(12) g$ can only be a transposition, it suffices to consider cases when $g$ equals to $e,(13), (23)$.  Therefore, equation \eqref{jopa} takes form:

\begin{multline}\label{d2}
(d_2)_v^u= \\=
\begin{cases}
\varphi_2(12)=2 &\text{ if }u=e, v=(12), g=e,\\
\varphi_2(gvg^{-1})=\varphi_2((13)(23)(13))=\varphi_2((12))=2&\text{ if }u=e, v=(23), g=(13),\\
\varphi_2(gvg^{-1})=\varphi_2((23)(13)(23))=\varphi_2((12))=2&\text{ if }u=e, v=(13), g=(23),\\
0,& \text{else }
\end{cases}=\\
=\begin{cases}
2, & u=e, v\in\{ (12),(23), (31)\},\\
0, & else
\end{cases}
\end{multline}

Finally, the theorem \eqref{explicit_map} tells us that the $\Z_4$-module of outer derivations admits the description in terms of generators and relations:
\begin{equation*}
\Out=<S|T>,
\end{equation*}
with the generating set equal to $$S=\{d_1+\Inn, d_2+\Inn\}$$ and the relations are 
$$T=\{2(d_1+\Inn)=\Inn=2(d_2+\Inn)\}.$$

\subsubsection*{Central derivations of $\Z_4[S_3]$}
%link

Let $z\in Z(G)$, $\tau:G\to A$ - additive homomorphism.

Linear operator $d_\tau^z$ defined on generators as $d_\tau^z: x\mapsto \tau(x)xz$, is a central derivation (see \cite{Arutyunov_2}). Any linear combination of such operators is also a central derivation.

\begin{proposition}
Central derivations are not isomorphic to outer derivations.
\end{proposition}

\begin{proof}
In our case, $Z(S_3)=\{e\}\implies$ only operators such as $d_\tau^e : x\mapsto \tau(x)xe=\tau(x)x$ span $ZDer$. The corresponding outer derivation is $d_\tau^e+\Inn$. 

We can see that there is no such homomorphism $\tau$,that $d_\tau^e+\Inn =d_2+\Inn$, because characters of $d_\tau^e$ and $d_2$  are not trivial on loops, but have supports in different subgroupoids.
\end{proof}
% (In detail, for any $\tau$ the support of $d_\tau^e$ lies in the subgroupoid $\Gamma_{[e]}$, and $d_2$ is supported on the subgroupoid $\Gamma_{[(12)]}$.)

\subsection{Example: derivations of $\F_{2^m}D_{2n}$}
\label{ex-F2m-D2n}
Here the $\F_{2^m}$ stands for a finite field of order $2^m$. 

The main goal of this section is to reproduce the result of paper \cite{Creedon-Hughes}. In order to calculate derivations of $\F_{2^m}D_{2n}$, we use several properties of dihedral groups, proofs of which can be found in expository papers by Keith Conrad \cite{Conrad3, Conrad4}.

To begin, the dihedral group of size $2n$ can be described by its presentation:  $$D_{2n}=\langle r,s|r^{2n}=s^2=(rs)^2=1 \rangle$$
 %???
The following theorem is proved in \cite{Conrad3}, (see Theorem 4.1).
\begin{theorem}
The group $D_{2n}$ has
  \begin{itemize}
\item two conjugacy classes of size 1:           $$\{1\},\{r^n\},$$
\item n-1 conjugacy classes of size 2: $$\{r,r^{-1}\},\{r^{2},r^{-2}\},\ldots,\{r^{ n-1},r^{- n+1}\},$$
\item two conjugacy classes of size n: $$\{r^{2i}s|0\leq i\leq n-1\},\{r^{2i+1}s|0\leq i\leq n-1\}.$$
  \end{itemize}
  
\end{theorem}
As a result, there are $n+3$ conjugacy classes.
  
  Let us choose and fix first element in any of conjugacy classes written above. Now we need to compute centralizers.
  
\begin{proposition} In $D_{2n}$ we get the following centralizers:
  \begin{enumerate}
\item  $Z(r^i)=\{1,r,\ldots, r^{2n-1}\}, 1\leq i\leq n,$
\item $Z(s)=\{1,s\},$
\item $Z(r^is)=\{1,r^i s\}.$
  \end{enumerate}
 \end{proposition}
  \begin{proof}
%  All elements in $D_{2n}$ fall into two categories: rotations $\{1, \ldots, r^{2n-1}\}$ and reflections $\{s, rs, \ldots, r^{2n-1}s\}$. We know that $srs=r^{-1}$ and $sr^{k}s^{-1}=r^{-k}$. Then, the computations are fairly straightforward: 
 %   \begin{enumerate}

%\item  In our case $i\geq 1$. If $s\in Z(r^i)$, then $r^{i}s=sr^i$, which is false.

%If $r^j s\in Z(r^i)$, then $r^{i}r^j s=r^j sr^i\implies r^{i}s=sr^i$, which is false.
%\item If $r^j\in Z(s), j>0$, then $r^{j}s=sr^j$, which is false. 

%If $r^j s\in Z(s), j>0$, then $r^j ss=r^j = s r^j s\implies r^js=sr^j$, which is false.
%\item If $s\in Z(r^i s)$ or $ r^j\in Z(r^i s)$, then we have the same contradictions. However, let $ r^j s\in Z(r^i s)$. Then $r^i s r^j s=r^i r^{-j}=r^{i-j}$ and $r^j s r^i s=r^j r^{-i}=r^{j-i}$. These coincide only in case $i=j$
 % \end{enumerate}
 The proposition is proved by direct computation
  \end{proof}
 
Now we are able to describe additive homomorphisms from centralizers to $\F_{2^m}$.

  \begin{enumerate}

\item  $\Hom_{Ab}(Z(r^i),\F_{2^m})=\Hom_{Ab}(\Z_{2n},\F_{2^m}), 1\leq i\leq n$;
\item $\Hom_{Ab}(Z(s))=\Hom_{Ab}(\Z_{2},\F_{2^m})$;
\item $\Hom_{Ab}(Z(r^is),\F_{2^m})=\Hom_{Ab}(\Z_{2},\F_{2^m})$.
  \end{enumerate}
  
 \begin{corollary}
 Using the decomposition theorems above,
\begin{enumerate}
\item  \begin{equation}
 \Out(\F_{2^m}D_{2n})\cong \bigoplus_{i=1}^{n+1} \Hom_{Ab}(\Z_{2n},\F_{2^m}) \oplus  \bigoplus_{i=1}^{2}\Hom_{Ab}(\Z_{2},\F_{2^m}),
 \end{equation}
\item  \begin{equation}
 \Inn(\F_{2^m}D_{2n})\cong \F_{2^m}^{4n-(n+3)}=\F_{2^m}^{3n-3}.
 \end{equation}
\end{enumerate}
\end{corollary}

 The lase formula can be explained in terms of number of conjugacy classes: the dimension of $\Inn$ is equal to number of elements in $D_{2n}$ minus number of conjugacy classes.

% \section{Appendix}
% There is an alternative proof of formulas \eqref{d1},\eqref{d2}, which uses the character technique without explicitly using theorem \eqref{explicit_map}. 
% \begin{proof}

% %bullshit starts here

% Call loops $\alpha$ and $\beta$ conjugated ($\alpha \sim \beta$), if $ \exists \varphi: \alpha=\varphi^{-1}\circ\beta\circ\varphi$.
% Because of the lemma \eqref{char_condition}, only morphisms $(u,v): ord(v)=2$ can have non-trivial characters. Let us look closely at morphisms which can have non-trivial characters:

% \end{proof}
% \begin{center}
% \begin{tabular}{ |c|c|c|c| } 
%  \hline
% \text{conjugacy class}&[e]& [(12)]&[(123)]\\
% \text{conjugated loops}&$\begin{aligned}[t]
%  ((12),(12))\sim \\
%        \sim((23),(23)) \sim\\
%        \sim((13),(13))     &
%     \end{aligned} $&$\begin{aligned}[t]
%  (e,(12))\sim \\
%        \sim(e,(23)) \sim\\
%        \sim(e,(13))     &
%     \end{aligned} $&~---~\\
% \text{derivation coefficients}&$d_{(12)}^{(12)}=d_{(23)}^{(23)}=d_{(13)}^{(13)}$&$d_{(12)}^{e}=d_{(23)}^{e}=d_{(13)}^{e}$&~---~\\
% \text{their possible values}&0,2&0,2&~---~\\
%  \hline
% \end{tabular}
% \end{center}

\end{document}